\documentclass[reqno,12pt]{amsart}
\usepackage{mathrsfs}
\usepackage{geometry}
\usepackage{epstopdf}
\usepackage{float}
\usepackage{graphicx}
\usepackage{amsfonts,color,amsmath,amssymb,amsthm,fancyhdr}
\usepackage{multirow}

\newcommand{\mc}{\mathcal}
\newcommand{\tr}{\textup{Tr}}
\newcommand{\ra}{\rangle}
\newcommand{\la}{\langle}

\newcommand{\F}{{\mathbb F}}

\newcommand{\cP}{{\mathcal P}}
\newcommand{\cM}{{\mathcal M}}

\newcommand{\cA}{{\mathcal A}}

\newcommand{\bO}{{\mathbf O}}
\newcommand{\bS}{{\mathbf S}}
\newcommand{\bU}{{\mathbf U}}

\newcommand{\GL}{\textup{GL}}

\newcommand{\SL}{\textup{SL}}
\newcommand{\SU}{\textup{SU}}

\newtheorem{thm}{Theorem}

\newtheorem{lemma}[thm]{Lemma}

\numberwithin{equation}{section}
\numberwithin{thm}{section}
\numberwithin{table}{section}

\newtheorem{remark}[thm]{Remark}

\renewcommand{\paragraph}{\roman{paragraph}}

\usepackage{enumerate}
\usepackage{graphics}
\usepackage{graphicx}
\usepackage{epsfig}
\usepackage{bm}
\usepackage{url}
\usepackage{float}
\usepackage{color}

\begin{document}

\title{A new family of $(q^4+1)$-tight sets with an automorphism group $F_4(q)$}
\date{}
\author{ Tao~Feng$^\ast$, Weicong~Li, Qing~Xiang$^\dagger$}
\thanks{$^\ast$Research partially supported by National Natural Science Foundation of China under Grant No. 12171428, 12225110}
\thanks{$^\dagger$Research partially supported by the National Natural Science Foundation of China Grant No. 12071206, 12131011, 12150710510, and the Sino-German Mobility Programme M-0157}

\address{Tao Feng, School of Mathematical Sciences,   Zhejiang University, Hangzhou, China.
}\email{tfeng@zju.edu.cn;}

\address{Weicong Li,  Department of Mahtematics, School of Sciences, Great Bay University, Dongguan, China.}
\email{liwweicong@gbu.edu.cn;}
\address{Qing Xiang,  Department of Mathematics and Shenzhen International Center for Mathematics, Southern University of Sicence and Technolohgy, ShenZhen, China.}
\email{xiangq@sustech.edu.cn.}
\maketitle

\begin{abstract}
In this paper, we construct a new family of $(q^4+1)$-tight sets in $Q(24,q)$ or $Q^-(25,q)$ according as $q=3^f$  or $q\equiv 2\pmod 3$. The novelty of the construction is the use of the action of the exceptional simple group $F_4(q)$ on its minimal module over $\F_q$.

\medskip
\noindent{{\it Keywords\/}: Exceptional simple group, intriguing set, polar space, tight set}

\smallskip

\noindent {{\it MSC(2020)\/}: 51E20, 05B25, 51A50}

\end{abstract}

\section{Introduction}

Let ${\mathcal P}_r$ be a finite classical polar space of rank $r\ge 2$. A tight set is a subset ${\mathcal T}$ of points of ${\mathcal P}_r$ such that for all points $P$ of ${\mathcal P}_r$, the intersection sizes $|P^{\perp}\cap {\mathcal T}|$ take exactly two values involving a parameter $i\ge 1$, according as $P\in {\mathcal T}$ or not, where $P^{\perp}$ is the set of points in ${\mathcal P}_r$ that are collinear with $P$. Tight sets in finite generalized quadrangles were first introduced by Payne \cite{Payne87}, and this notion has been generalized to finite classical polar spaces by Drudge \cite{DrudgeThesis}. Parallel to the notion of a tight set is the concept of an $m$-ovoid. The notions of tight sets and $m$-ovoids were unified under the umbrella of {\it intriguing sets} in \cite{BambergTightsets2007}.  During the past two decades, intriguing sets have been extensively investigated because of their close connections with many other combinatorial/geometric objects such as strongly regular graphs, partial difference sets, Boolean degree one functions, and Cameron-Liebler line classes, cf. \cite{BambergTightsets2007,FMRXZ2021,FI2019Boolean}.

In general,  most known examples of tight sets are found in finite polar spaces of low ranks. However, 1-tight sets exist in all finite classical polar space ${\mathcal P}_r$:  every generator $\mc{G}$ of ${\mathcal P}_r$  is a 1-tight set since $|P^{\perp}\cap \mc{G}|=|\mc{G}|$ if $P\in \mc{G}$, and $|P^{\perp}\cap \mc{G}|=$ the size of a hyperplane of $\mc{G}$ if $P\not\in \mc{G}$. In fact, Drudge \cite{DrudgeThesis} proved that any 1-tight set of ${\mathcal P}_r$ must be a generator. Tight sets with large parameters are much more complicated.  One of the main problems concerning tight sets is to determine for which values of $i$ there exist $i$-tight sets in ${\mathcal P}_r$. A second problem is to characterize $i$-tight sets in ${\mathcal P}_r$  for specific parameter $i$.

As we saw before, one can deduce fairly easily that a generator of ${\mathcal P}_r$  is a 1-tight set. Suppose that $A$ and $B$ are an $i$-tight set and a $j$-tight set in ${\mathcal P}_r$, respectively. If $A, B$ are disjoint, then $A\cup B$ is  an $(i+j)$-tight set; if $A\subseteq B$, then $B\setminus A$ is a $(j-i)$-tight set. Thus, the union of any $i$ pairwise disjoint generators (i.e., the elements of a partial spread) of ${\mathcal P}_r$ forms an $i$-tight set. For this reason, the first question has a simple and complete answer if the classical polar space $\cP_r$ admits a spread of generators. Thus, we will only consider the problem of constructing tight sets for polar spaces which do not admit a spread; and for such a polar space,  it is interesting to construct tight sets with large parameters $i$, where $i$ is greater than the maximum size of a partial spread.  Some results on the existence of spreads of classical polar spaces can be found in \cite[Table 7.4]{GGG}.

In this paper, we construct a new family of $(q^4+1)$-tight sets in $Q(24,q)$ or $Q^-(25,q)$ according as $q=3^f$  or $q\equiv 2\pmod 3$. Our main theorem in this paper is the following.
\begin{thm}\label{thm_main}
 Let $q=p^f$, where $p$ is prime and $f\ge 1$. If  $q\equiv2\pmod 3$, then the classical polar space $Q^-(25,q)$ admits a $(q^4+1)$-tight set with an automorphism group isomorphic to the exceptional group $F_4(q)$; if $q=3^f$, then the classical polar space $Q(24,q)$ admits a $(q^4+1)$-tight set with an automorphism group isomorphic to the exceptional group $F_4(q)$.
\end{thm}

The novelty of the construction is the use of the action of the exceptional simple group $F_4(q)$ on its minimal module over $\F_q$. We remark that the characteristic $3$ case of Theorem \ref{thm_main} is essentially due to  Cohen and Cooperstein. They showed in \cite[Table~2]{Cohen1988} that the group $F_4(q)$ has two orbits on the singular points of the polar space $Q(24,q)$ associated with its $25$-dimensional minimal module, and from this fact we deduce that each orbit is a tight set.

The paper is organized as follows. In Section 2, we introduce some preliminary results on intriguing sets, octonions and the minimal module of $F_4(q)$. In particular, we give a more effective way to decide whether a subset $\cM$ of points of a finite classical polar space ${\mathcal P}_r$ is an intriguing set, see Lemma \ref{lem_basic}.  In Section 3, we present the proof of Theorem \ref{thm_main}.

\section{Preliminary}

\subsection{Intriguing sets on finite classical polar spaces}
Let $\F_{q}$ be the finite field of order $q$, where $q=p^f$, $p$ is a prime and $f\ge 1$.  Let $V$ be a $d$-dimensional vector space over $\F_q$ equipped with a nondegenerate reflexive sesquilinear form or quadratic form $\kappa$, and let $\cP_r$ be the associated polar space. A point of $\cP_r$ is defined as a $1$-dimensional totally isotropic/singular subspace of $V$. A maximal totally isotropic/singular subspace of $\cP_r$ is called a \textit{generator} of $\cP_r$. All generators have the same dimension $r$, which is called the \textit{rank} of $\cP_r$. A generator of $\cP_r$ has $\frac{q^r-1}{q-1}$ points. An \textit{ovoid}  of $\cP_r$ is a set of points which meets each generator in exactly one point. We use $\theta_r$ to denote the size of a putative ovoid, which we call {\it the ovoid number} of $\cP_r$.  A simple counting arguement shows that $|\cP_r|=\theta_r\cdot \frac{q^r-1}{q-1}$.   We list the ranks and the ovoid numbers of the six classical polar spaces in Table \ref{tab_valuede}.
\begin{table}[!h]\footnotesize\tabcolsep 16pt
	\centering
	\caption{The parameters $r$ and $\theta_{r}$}	\label{tab_valuede}
	\begin{tabular}{cccccc}
		\hline
		& $d$ & $f$& polar space $\cP_r$ & rank $r$ & ovoid number $\theta_{r}$ \\ 
		\hline
		$\bS$ & even&-& $W(d-1,q)$ & $d/2$ & $q^{d/2}+1$ \\ 
		\hline
		\multirow{3}*{$\bO$}
		&even &-& $Q^+(d-1,q)$ & $d/2$ & $q^{d/2-1}+1$ \\ 
		& even &-& $Q^-(d-1,q)$ & $d/2-1$ & $q^{d/2}+1$ \\ 
		& odd &-& $Q(d-1,q)$ & $(d-1)/2$ & $q^{(d-1)/2}+1$\\ 
		\hline
		\multirow{2}*{$\bU$}
		& odd &even &$H(d-1,q)$ & $(d-1)/2$ & $q^{d/2}+1$\\ 
		& even &even& $H(d-1,q)$ & $d/2$ & $q^{(d-1)/2}+1$\\ 

\hline
	\end{tabular}
\end{table}	

Suppose that $r\ge 2$, and let $\cM$ be a nonempty set of  points of $\cP_r$. The set $\cM$ is called an \textit{intriguing set} if there exist some constants $h_1\ne h_2$ such that $|P^{\perp}\cap\cM|=h_{1}$ or $h_{2}$ according as $P\in\cM$ or not, where $P$ ranges over all the points of $\cP_r$, cf. \cite{BambergTightsets2007}. An intriguing set $\cM$ is proper if $\cM\ne\cP_r$. There are exactly two types of intriguing sets:
\begin{enumerate}
	\item[(1)]$i$-tight sets: $|\cM|=i\cdot\frac{q^r-1}{q-1}$, $h_1=q^{r-1}+i\cdot\frac{q^{r-1}-1}{q-1}$, $h_2=i\cdot\frac{q^{r-1}-1}{q-1}$, and
	\item[(2)]$m$-ovoids: $|\cM|=m\theta_r$, $h_1=(m-1)\theta_{r-1}+1$, $h_2=m\theta_{r-1}$.
\end{enumerate}
We refer the reader to \cite{BambergTightsets2007} for more properties of intriguing sets. In particular, if $H$ is a subgroup of semisimilarities that has exactly two orbits $O_1,\,O_2$ on the points of $\cP_r$, then both $O_1$ and $O_2$ are intriguing sets of the same type.

To prove that a candidate subset $\cM$ of  points of $\cP_r$ is an intriguing set, one needs to show that $\cM$  is a two-intersection set with respect to the perp of singular points, cf. \cite{BambergTightsets2007}. However, there is one shortcut method which seems to have gone unnoticed (see Leema 2.1 below).  We observe that $\theta_r-1=q(\theta_{r-1}-1)$ by Table \ref{tab_valuede}.

\begin{lemma}\label{lem_basic}
Let $\cM$ be a subset of size $i \frac{q^r-1}{q-1}$ or $m \theta_r$ in $\cP_r$ for some $i$ or $m$, and let $h_1,h_2$ be the corresponding parameters determined by $|\cM|$ (see above).  Then the following are equivalent:
\begin{itemize}
\item [(1)] $\cM$ is an intriguing set in $\cP_r$;
\item [(2)] $|P^\perp\cap\cM|=h_1$ for all $P\in\cM$;
\item [(3)] $|P^\perp\cap\cM|=h_2$ for all $P\in \cP_r\setminus \cM$.
\end{itemize}
\end{lemma}
\begin{proof}
We will prove the equivalence of (1) and (2). It is clear that (1) implies (2). So we only need to prove that (2) implies (1). We will use $\cP_j$ to denote the polar space of rank $j$ and of the same type as $\cP_r$.  We compute $\sum_{P\in\cP_r}|P^\perp\cap\cM|$ and $\sum_{P\in\cP_r}|P^\perp\cap\cM|^2$, and will show that $\sum_{P\in\cP_r\setminus\cM}\left(|P^\perp\cap\cM|-h_2\right)^2=0$. First we have
\begin{align*}
	\sum_{P\in\cP_r} |P^\perp \cap \cM| &= |\{ (P,z):  P\in \cP_r, z\in \cM, P\sim z\}|\\
	&= \sum_{z\in \cM} |z^\perp\cap \cP_r|= q |\cM| \cdot |\cP_{r-1}|+|\cM|.
\end{align*}
For the last equality, we observe that for a point $z\in\cM$ the intersection $z^\perp\cap \cP_r$ is a cone with vertex $z$ and base $\cP_{r-1}$.
Furthermore we have
\begin{align}
\nonumber \sum_{P\in\cP_r}|P^\perp\cap\cM|^2&=\sum_{P\in \cP_r}|\{(P,z_1,z_2):\,z_1,z_2\in\cM,\, z_1\sim P,\, z_2\sim P\}|\\
&=\sum_{z\in\cM}|z^\perp\cap\cP_r|+\sum_{\substack{z_1,z_2\in\cM,\\ z_1\ne z_2}}|\la z_1,z_2\ra^\perp\cap\cP_r|.\label{eqn_PcapM}
\end{align}
For the last summation in \eqref{eqn_PcapM}, each summand takes value either $q^2 \cdot |\cP_{r-2}|+q+1$ or  $|\cP_{r-1}|$  depending on whether $z_1,z_2$ are perpendicular or not: $\la z_1,z_2\ra^\perp\cap\cP_r$ is a cone with vertex $\la z_1,z_2\ra$ and base $\cP_{r-2}$ if $\la z_1,z_2\ra$ is totally singular or isotropic, and it is $\cP_{r-1}$ otherwise. 
 Next we compute directly
 \begin{align}
 	\notag \sum_{P\in\cP_r}\left(|P^\perp\cap\cM|-h_2\right)^2&=\sum_{P\in\cP_r}\Big(|P^\perp\cap\cM|^2-2h_2|P^\perp \cap \cM|+h_2^2\Big)\\
 	\notag &=|\cM|\cdot \Big( |\cP_{r-1}|(|\cM|+q-h_1-2h_2q)+ (qh_1-q+h_1-2h_2)\\ &
 	 \ +q^2(h_1-1) |\cP_{r-2}|+\frac{h_2^2|\cP_r|}{|\cM|} \Big). \label{eqn_PcapMh2}
 \end{align}
 First consider the case where $\cM$ is an $i$-tight set, so that  $|\cM|=i\frac{q^{r}-1}{q-1},h_1=i\frac{q^{r-1}-1}{q-1}+q^{r-1},h_2=i\frac{q^{r-1}-1}{q-1}$.  Observe that  the right hand side of $\eqref{eqn_PcapMh2}$ divided by $|{\mathcal M}|$ can be viewed as a degree-one polynomial in variable $i$. Then some tedious but routine computations show that the coefficient of $i$ equals $0$. Hence we are only concerned with the constant term of the aforementioned polynomial, which equals
\begin{align*}
    &|\cM| \big(  |\cP_{r-1}| (q-q^{r-1}) +(q+1)q^{r-1}-q+q^2(q^{r-1}-1)|\cP_{r-2}|\big)\\ =&q^{2r-2} =|\cM|\cdot (h_1-h_2)^2.
\end{align*}
From $\sum_{P\in\cP_r}\left(|P^\perp\cap\cM|-h_2\right)^2=|\cM|(h_1-h_2)^2$ we deduce that $\sum_{P\in\cP_r\setminus\cM}\left(|P^\perp\cap\cM|-h_2\right)^2=0$. This proves that (2) implies (1) when $\cM$ is an $i$-tight set.  When $\cM$ is an $m$-ovoid, we can use the same method to obtain the result. The proof of (1)$\Leftrightarrow$(3) is similar  and we omit the details. This completes the proof.
\end{proof}
\subsection{Octonions and the minimal module of $F_4(q)$}
The octonion algebra $\mathbb{O}_{\F_q}$ is an 8-dimensional non-commutative and non-associative algebra over $\F_q$.  A detailed introduction to octonions can be found in the book \cite{Springer2000}. Let $\{x_1,\ldots,x_8\}$ be the basis of the octonion algebra $\mathbb{O}$ (here in order to simplify notation we omit the subscript) as defined in \cite[(4.26)]{wilson2009finite}. With respect to this basis, the multiplication of $\mathbb{O}$ is defined as follows:
\begin{equation}\label{eqn_Omult}
	\begin{tabular}{c|cccccccc|}
		 & $x_1$ &$x_2$ &$x_3$ &$x_4$ &$x_5$ &$x_6$ &$x_7$ & $x_8$\\ \hline
       $x_1$ & & & & & $x_1$ & $x_2$ &$-x_3$ &$-x_4$ \\
       $x_2$ & & & $- x_1$ &$x_2$ &  &  & $-x_5$& $x_6$ \\
       $x_3$ & & $x_1$ & & $x_3$ &  & $-x_5$ & & $x_7$ \\
       $x_4$ & $x_1$ & & & $x_4$ & & $x_6$ & $x_7$ & \\
       $x_5$ & & $x_2$ & $x_3$ & & $x_5$ & & & $x_8$ \\
       $x_6$ & $-x_2$ & & $-x_4$ & &$x_6$ & & $x_8$ & \\
       $x_7$ & $x_3$ & $-x_4$ & & & $x_7$ & $-x_8$ & & \\
       $x_8$ & $-x_5$ & $-x_6$ & $x_7$ & $x_8$ & & & & \\ \hline 	
 	\end{tabular}
\end{equation}
where the blank entries are $0$, cf. \cite[(4.27)]{wilson2009finite}. It is clear that $\mathbf{1}:=x_4+x_5$ is the identity of $\mathbb{O}$.  The octonion conjugation $^-$ swaps $x_4, \,x_5$ and maps the other $x_i$'s to their negations, and it is an anti-isomorphism of $\mathbb{O}$.  For an octonion $x=\sum_{i=1}^8\lambda_ix_i$, we define its trace as $\textup{Tr}(x)=x+\overline{x}=\lambda_4+\lambda_5$ and its norm as $N(x)=x\overline{x}$. It is routine to show that $N(x)=\sum_{i=1}^4\lambda_i\lambda_{9-i}$ which takes value in $\F_q$, and defines a nondegenerate hyperbolic quadric $Q^+(7,q)$.  Its associated bilinear form is $B(x,y):=N(x+y)-N(x)-N(y)$, and we have $B(x,y)=\tr(x\overline{y})$ for $x,y\in\mathbb{O}$. In particular, we have $\tr(x)=B(x,1)$. The octonion algebra $\mathbb{O}$ is an alternative ring by \cite[Lemma~3]{wilsonF4}, so that the subalgebra generated by any two elements is associative, cf. \cite[Chapter~1.4]{Springer2000}. It follows that $\mathbb{O}$ is a composition algebra, i.e., $N(xy)=N(x)N(y)$ for $x,y\in\mathbb{O}$. The automorphism group of the octonion algebra $\mathbb{O}$ is known as $G_2(q)$, and each of its element fixes the vector $\mathbf{1}:=x_4+x_5$. It holds that $x^2-\tr(x)x+N(x)=0$ for $x\in\mathbb{O}$ by \cite[Proposition 1.2.3]{Springer2000}, so $G_2(q)$ lies in the isometry group of the quadratic form $N$. We write $\mathbf{1}^\perp$ for the perp of $\mathbf{1}$ with respect to $B$, which is stabilized by $G_2(q)$.\medskip

The $G_2(q)$-orbits on the nonzero octonions are implicitly known by the results in \cite[Chapter 4.3]{wilson2009finite}. We list them explicitly in the next lemma for future use.
\begin{lemma}\label{lem_vt}
Suppose that $q>2$. For $q$ odd, let $\alpha$ be a nonsquare of $\F_q^*$ and $T$ be a complete set of coset representatives of $\{1,-1\}$ in $\F_q^*$; for $q$ even, let $\beta$ be an element of $\F_q^*$ with absolute trace $1$ and and $S$ be a complete set of coset representatives of $\{0,1\}$ in $\F_q$. Then the $G_2(q)$-orbits on the octonions are as listed in Table \ref{tab_octOrb}.
\end{lemma}

\begin{table}[h]
\centering\caption{The $G_2(q)$-orbits on the octonions}\label{tab_octOrb}
\begin{tabular}{|c|c|c|c|c|c|}
\hline
Representative                                                                               & Orbit size & Stabilizer           & Trace & Norm                & Condition                 \\ \hline
$k$, $k\in\F_q$                                                                            & $1$        & $G_2(q)$             & $2k$  & $k^2$               & \multirow{2}{*}{}         \\ \cline{1-5}
$x_1+k$, $k\in\F_q$                                                                          & $q^6-1$    & $q^{2+1+2}:\SL_2(q)$ & $2k$  & $k^2$               &                           \\ \hline
\begin{tabular}[c]{@{}c@{}}$a(x_4-x_5)+k$, \\ $a\in T,k\in\F_q$\end{tabular}                 & $q^6+q^3$  & $\SL_3(q)$           & $2k$  & $k^2-a^2$           & \multirow{2}{*}{$q$ odd}  \\ \cline{1-5}
\begin{tabular}[c]{@{}c@{}}$a(x_1-\alpha x_8)+k$, \\ $a\in T,k\in \F_q$\end{tabular}         & $q^6-q^3$  & $\SU_3(q)$           & $2k$  & $k^2-\alpha a^2$    &                           \\ \hline
\begin{tabular}[c]{@{}c@{}}$a(x_4+k)$, \\ $a\in\F_q^*,k\in S$\end{tabular}                     & $q^6+q^3$  & $\SL_3(q)$           & $a$   & $a^2(k^2+k)$            & \multirow{2}{*}{$q$ even} \\ \cline{1-5}
\begin{tabular}[c]{@{}c@{}}$a(x_4+x_1+\beta x_8+k)$, \\ $a\in\F_q^*$, $k\in S$\end{tabular} & $q^6-q^3$  & $\SU_3(q)$           & $a$   & $a^2(k^2+k+\beta)$ &                           \\ \hline
\end{tabular}
\end{table}
\begin{proof}
We give a sketch of the proof by quoting the relevant results in \cite[Chapter 4.3]{wilson2009finite}. Since each element of $G_2(q)$ fixes $\mathbf{1}$, each element $k\in\F_q$ is stabilized by $G_2(q)$. Take an element $k\in\F_q$ and set $v=x_1+k$. For an element $g\in G_2(q)$, it stabilizes $v$ if and only if it stabilizes $x_1$. The stabilizer of $\la x_1\ra$ is a maximal parabolic subgroup $q^{2+1+2}:\GL_2(q)$, and the stabilizer of $x_1$ in the latter subgroup is $q^{2+1+2}:\SL_2(q)$. Hence the $G_2(q)$-orbit of $x_1+k$ has size $\frac{|G_2(q)|}{q^5\cdot|\SL_2(q)|}=q^6-1$ for each $k\in\F_q$. The trace and norm of $x_1+k$ are respectively $2k,k^2$, so \begin{color}{blue} the $(x_1+k)$'s \end{color} are in different $G_2(q)$-orbits for  distinct $k$'s in $\F_q$. The arguments so far work for both even and odd $q$.

Suppose that $q$ is odd and set $u=x_4-x_5$, $v=x_1-\alpha x_8$, where $\alpha$ is a nonsquare in $\F_q^*$. The stabilizer of $\la u\ra$ in $G_2(q)$ is a maximal subgroup $\SL_3(q):\la s\ra$, where $s$ has order $2$ and swaps $x_4$ and $x_5$. The stabilizer of $u$ in the latter subgroup is $\SL_3(q)$. As in the previous paragraph, we deduce that the $G_2(q)$-orbit of $au+k$ has size $\frac{|G_2(q)|}{|\SL_3(q)|}=q^6+q^3$ for each $a\in\F_q^*$ and $k\in\F_q$. The trace and norm of $au+k$ are respectively $2k,-a^2+k^2$, so $au+k$ and $a'u+k'$ are in the same $G_2(q)$-orbit if and only if $a'=\pm a$ and $k'=k$. The subalgebra $C:=\la \mathbf{1},v\ra$ is isomorphic to $\F_{q^2}$, and the stabilizer of $v$ in $G_2(q)$ is $\SU_3(q)$. There is an involution $r\in G_2(q)$ that maps $x_1,x_8$ to $-x_1,-x_8$ and thus $v$ to $-v$, and $\SU_3(q):\la r\ra$ is a maximal subgroup of $G_2(q)$ that stabilizes $\la v\ra$. Similarly, we deduce that the $G_2(q)$-orbit of $av+k$ has size $\frac{|G_2(q)|}{|\SU_3(q)|}=q^6-q^3$ for each $a\in\F_q^*$ and $k\in\F_q$. The trace and norm of $av+k$ are respectively $2k,k^2-a^2$, so $av+k$ and $a'v+k'$ are in the same $G_2(q)$-orbit if and only if $a'=\pm a$ and $k'=k$. Since
\[
q+q\cdot (q^6-1)+ \frac{1}{2}(q-1)q\cdot (q^6+q^3)+ \frac{1}{2}(q-1)q\cdot (q^6-q^3)=q^8,
\]
we conclude that we have obtained all the $G_2(q)$-orbits and their relevant information is as listed in Table \ref{tab_octOrb} for $q$ odd.

Suppose that $q$ is even and set $u=x_4$, $v=x_1+\beta x_8+x_4$, where $\beta$ has absolute trace $1$, i.e., $X^2+X+\beta$ is irreducible over $\F_q$. The element $v$ has trace $1$ and norm $\beta$, and there are $q^6-q^3$ octonions with those properties. The stabilizer of $u$ in $G_2(q)$ is $\SL_3(q)$, and there is an order $2$ element $s$ that swaps $u_4,u_5$ such that $\SL_3(q):\la s\ra$ is a maximal subgroup. We have $v^2+v+\beta=0$, so the subalgebra $\la \mathbf{1},v\ra$ is a field with $q^2$ element. There is an involution $r\in G_2(q)$ that fixes $x_1,x_8$ and swaps $x_4$ and $x_5$, so that $r$ maps $v$ to \begin{color}{blue}$v+\mathbf{1}$\end{color}. The stabilizer of $v$ in $G_2(q)$ is $\SU_3(q)$, and $\SU_3(q):\la r\ra$ is a maximal subgroup. The remaining arguments are exactly the same as in the $q$ odd case, and we omit the details. This completes the proof.
\end{proof}

For a nonzero element $a$ in $\mathbb{O}$, we define its left and right annihilators as follows
\begin{align*}
\textup{ann}_L(a)=\{x\in\mathbb{O}:\,xa=0\},\quad \textup{ann}_R(a)=\{x\in\mathbb{O}:\,ax=0\}.
\end{align*}
By \cite[Lemma 1.3.3]{Springer2000}, we have $x(\overline{x}a)=N(x)a$ and $(ax)\overline{x}=N(x)a$, so both $\textup{ann}_L(a)$ and $\textup{ann}_R(a)$ are totally singular subspaces for the quadratic form $N$. If $N(a)\ne 0$, then both of its annihilators are trivial by a similar argument.

\begin{lemma}\label{lem_condO}
Suppose that $D,E$ are nonzero elements in $\mathbb{O}$ such that $DE=0$. Then $D\overline{D}=E\overline{E}=0$, $\textup{ann}_L(D)$ and $\textup{ann}_R(D)$ are totally isotropic subspaces of dimension $4$, and
\begin{itemize}
\item[(1)] $\textup{ann}_L(D)\cap \textup{ann}_R(D)$ has dimension $3$ or $1$  according as $D\in \mathbf{1}^\perp$ or not;
\item[(2)] $\textup{ann}_L(D)\cap \textup{ann}_R(E)$ has dimension $3$.
\end{itemize}
\end{lemma}
\begin{proof}
By the arguments preceding this lemma, we have $D\overline{D}=E\overline{E}=0$, and $\textup{ann}_L(D)$ and $\textup{ann}_R(D)$ are totally isotropic subspaces. First consider the case where $D$ is in $\mathbf{1}^\perp$. Since $G_2(q)$ is transitive on the nonzero singular vectors in $\mathbf{1}^\perp$, cf. \cite[p.~125]{wilson2009finite}, we assume that $D=x_1$ without loss of generality. It follows from \eqref{eqn_Omult} that $\textup{ann}_L(x_1)=\la x_1,x_2,x_3,x_5\ra $ and $\textup{ann}_R(x_1)=\la x_1,x_2,x_3,x_4\ra$. Both annihilators of $x_1$ have dimension $4$ and their intersection $\la x_1,x_2,x_3\ra$ has dimension $3$. Since $DE=0$, we deduce that $E=\sum_{i=1}^4a_ix_i$ for some $a_i$'s in $\F_q$. For an element $y\in \textup{ann}_L(D)$, we have $y=b_1x_1+b_2x_2+b_3x_3+b_5x_5$ for some $b_i$'s in $\F_q$. If further $y$ is in $\textup{ann}_R(E)$, then $Ey=(a_1b_5-a_2b_3+a_3b_2+a_4b_1)x_1=0$. There are $q^3$ such $(b_1,b_2,b_3,b_5)$ tuples, so $\textup{ann}_L(x_1)\cap \textup{ann}_R(E)$ has dimension $3$ as desired.

We next consider the case $D$ is not in $\mathbf{1}^\perp$. Since $\la x_4\ra$ and $\la x_5\ra$ are in the same $G_2(q)$-orbit and $D$ is singular, we assume without loss of generality that $D=x_4$ up to scalar by Table \ref{tab_octOrb}. We deduce from \eqref{eqn_Omult} that $\textup{ann}_L(x_4)=\la x_1,x_5,x_6,x_7\ra $ and $\textup{ann}_R(x_4)=\la x_2,x_3,x_5,x_8\ra$. Their intersection is $\la x_5\ra$ which has dimension $1$. Since $DE=0$, we deduce that $E=a_2x_2+a_3x_3+a_5x_5+a_8x_8$ for some $a_i$'s in $\F_q$. For an element $y\in \textup{ann}_L(D)$, we have $y=b_1x_1+b_5x_5+b_6x_6+b_7x_7$ for some $b_i$'s in $\F_q$. If further $y$ is in $\textup{ann}_R(E)$, then $Ey=(-a_2b_7-a_3b_6+a_5b_5-a_8b_1)x_5=0$. There are $q^3$ such $(b_1,b_2,b_3,b_5)$ tuples, so $\textup{ann}_L(x_4)\cap \textup{ann}_R(E)$ has dimension $3$ as desired. This completes the proof.
\end{proof}

The following result is well-known, and we include a short proof here. We refer the reader to \cite[Proposition~9.10]{Grove} for the case where $q$ is odd. For a property $\mathbf{P}$, we use the Iverson bracket $[\![\mathbf{P}]\!]$ which takes value $1$ or $0$ according as $\mathbf{P}$ holds or not.
\begin{lemma}\label{lem_Nka}
For a natural number $k$ and an element $a\in\F_q$, write $N_k(a)$ for the number of tuples $(a_1,\ldots,a_{2k})\in\F_q^{2k}$ such that $a_1a_2+\ldots+a_{2k-1}a_{2k}=a$. Then
\[
  N_k(a)=q^{2k-1}-q^{k-1}+q^k\cdot[\![a=0]\!]\quad\textup{ for } a\in\F_q.
\]
\end{lemma}
\begin{proof}
It is easy to see that the claim holds for $k=1$. We have
\begin{align*}
N_k(a)&=\sum_{a_{2k-1}\in\F_q}\sum_{a_{2k}\in\F_q}N_{k-1}(a-a_{2k-1}a_{2k})\\
&=q^2N_{k-1}(1)+(N_{k-1}(0)-N_{k-1}(1))\cdot|\{a_{2k-1},a_{2k}\in\F_q:\,a_{2k-1}a_{2k}=a\}|\\
&=q^2N_{k-1}(1)+(N_{k-1}(0)-N_{k-1}(1))\left(q-1+q\cdot[\![a=0]\!]\right),
\end{align*}
and the claim follows by induction on $k$.
\end{proof}

\begin{lemma}\label{lem_ONorm}
For $a\in\F_q$ there are $q^7-q^3+q^4\cdot[\![a=0]\!]$ octonions of norm $a$, and among them there are $q^6-q^3+q^4\cdot[\![a=0]\!]$ such $\alpha$'s that $\tr(x_1{\overline{\alpha}})=0$.
\end{lemma}
\begin{proof}
The first claim immediately follows by specifying $k=4$ in Lemma \ref{lem_Nka}. For an octonion $\alpha=\sum_{i=1}^8 \lambda_ix_i$, we have $\tr(x_1\overline{\alpha})=\lambda_8$ and $N(\alpha)=\sum_{i=1}^{4}\lambda_i\lambda_{9-i}$. If $\lambda_8=0$, then $N(\alpha)=\sum_{i=2}^{4}\lambda_i\lambda_{9-i}$. We deduce that there are $q\cdot N_3(a)$ such $\alpha$'s that satisfy $\lambda_8=0$ and  have norm $a$, where $N_3(a)$ is as in Lemma \ref{lem_Nka}. This completes the proof.
\end{proof}

\begin{lemma}\label{lem_DEpair}
There are $(q^6-q^3+q^4-1)(q^4-1)$ of $(D,E)$ pairs of nonzero octonions such that $DE=0$ and $\tr(x_1\overline{D})=0$.
\end{lemma}
\begin{proof}
If $DE=0$ for some nonzero $D,E$, then $D\overline{D}=0$ by Lemma \ref{lem_condO}. There are $q^6-q^3+q^4-1$ nonzero $D$'s such that  $D\overline{D}=0$ and $\tr(x_1\overline{D})=0$ by Lemma \ref{lem_ONorm}, and for each such $D$, there are $q^4-1$ nonzero $E$'s such that $DE=0$ by Lemma \ref{lem_condO}. The claim then follows.
\end{proof}

Let $\cA$ the the algebra of $3\times3$ Hermitian matrices over the octonions $\mathbb{O}$. To be specific, a $3\times 3$ matrix over $\mathbb{O}$ is an Hermitian matrix if $x^\top=\overline{x}$ and the diagonal entries of $x$ lie in $\F_q$. Here, $\overline{x}$ is the matrix obtained by applying the octonion conjugate to the entries of $x$. For $d,e,f\in\F_{q}$ and $D,E,F\in\mathbb{O}$, we define
\[
 (\lambda_0,\lambda_0',\lambda_0''\mid D,E,F):=\begin{pmatrix}
 	\lambda_0 & F  &  \overline{E} \\
\overline{F} & \lambda_0' &  D\\
 	E & \overline{D} &\lambda_0''
 \end{pmatrix}
\]
which is in $\cA$. For $u,v\in\cA$, their product in the algebra $\cA$ is $u\circ v:=uv+vu$, where $uv$ and $vu$ are matrix multiplications. We set $I:=(1,1,1\mid0,0,0)$, so that $I\circ a=2a$ for any element $a$ in $\cA$. The algebra $\cA$ has close connections with Albert algebras and Jordan algebras, cf. \cite{Mc_JA,Schaffer_JA}. We choose a basis $\{w_i,w_i',w_i'':\,0\le i\le 8\}$  of $\cA$ as follows:
\begin{align*}
  &w_0=(1,0,0\mid0,0,0)\textup{ and }w_i=(0,0,0\mid x_i,0,0)\textup{ for }i>0,\\
  &w_0'=(0,1,0\mid0,0,0)\textup{ and }w_i'=(0,0,0\mid 0,x_i,0)\textup{ for }i>0,\\
  &w_0''=(0,0,1\mid0,0,0)\textup{ and }w_i''=(0,0,0\mid 0,0,x_i)\textup{ for }i>0.
\end{align*}
The multiplication table of $\cA$ with respect to the basis $\{w_i,w_i',w_i'':\,0\le i\le 8\}$ can be written down explicitly as in \cite[(4.90)-(4.92)]{wilson2009finite}. The similarity group of the Dickson-Freudenthal determinant of $\cA$, which we denote by $\tilde{E}_6(q)$, is the universal covering group of the simple group $E_6(q)$, cf. \cite{wilsonF4}. It has three orbits on the nonzero vectors of $\cA$, which are called the white, gray and black vectors respectively, cf. \cite{Aschbacher1987,Cohen1988}. They correspond to vectors of rank $1$, rank $2$ and rank $3$ in \cite{Jacobson}. The $1$-dimensional subspace spanned by a white, gray, black vector is called a white, gray and black point respectively. The stabilizer of $I$ in $\tilde{E}_6(q)$ is $F_4(q)$.

For an element $v=\sum_{t=0}^8(\lambda_tw_t+\lambda_t'w_t'+\lambda_t''w_t'')\in \cA$,  we define the trace of $v$ as $\tr_{\cA}(v)=\lambda_0+\lambda_0'+\lambda_0''$. We define the $F_4(q)$-invariant subspaces $U=\{v\in \cA:\tr_{\cA}(v)=0\}$, $U'=\la I\ra_{\F_q}$, and set $W:=U/(U\cap U')$. For $v\in U$, we define
\[
  Q_0(v)=\lambda_0^2+\lambda_0\lambda_0'+\lambda_0'^2+\sum_{t=1}^4(\lambda_t\lambda_{9-t}+\lambda_t'\lambda_{9-t}'+\lambda_t''\lambda_{9-t}'').
\]
In particular, if $v=(\lambda_0,\lambda_0',\lambda_0''\mid D,E,F)$, then $Q_0(v)=\lambda_0^2+\lambda_0\lambda_0'+\lambda_0'^2+D\overline{D}+E\overline{E}+F\overline{F}$.
The quadratic form $Q_0$ on $U$ is $F_4(q)$-invariant and has $U\cap U'$ as its radical, so it induces a nondegenerate quadratic form $Q$ on $W$. The space $W$ is the minimal module of $F_4(q)$, and $\dim(W)=25$ or $26$ according as the characteristic $p$ of $\F_q$ is $3$ or not. The form $Q$ on $W$ is hyperbolic if $q\equiv 1\pmod{3}$ and is elliptic if $q\equiv 2\pmod{3}$.  The polar spaces in Theorem \ref{thm_main} are defined by the quadratic space $(W,Q)$.

\section{The proof of Theorem \ref{thm_main}}
In this section we give the proof of Theorem \ref{thm_main}. Take the same notation as in the previous section, and we consider the polar space defined by the quadratic space $(W,Q)$. If the characteristic $p$ of $\F_q$ is $3$, then the associated polar space is $Q(24,q)$ and the group $F_4(q)$ has two orbits on the singular points by \cite[Table~2]{Cohen1988}. One orbit $\cM_1$ has size $(q^4+1)\frac{q^{12}-1}{q-1}$ which is not divisible by the ovoid number $q^{12}+1$, so it is a $(q^4+1)$-tight set of $Q(24,q)$ as desired. We suppose that $q\equiv 2\pmod{3}$ in the sequel, so that $W=U$ and the associated polar space is $Q^-(25,q)$. Let $M_1$ be the set of white vectors in $U$, and write $\cM_1$ for the corresponding set of projective points. The set $M_1$ has size $(q^4+1)(q^{12}-1)$ and forms a single $F_4(q)$-orbit by \cite[(W.3)]{Cohen1988}. By \cite[Lemma~5]{wilsonF4}, we enumerate the white vectors in $U$ as follows:
\begin{enumerate}
\item[(I)] $v=f\cdot(A\overline{A},B\overline{B},1\mid\overline{B}, A,\overline{A}B)$ for some $f\in\F_q^*$ and $A,B\in\mathbb{O}$ such that $A\overline{A}+B\overline{B}+1=0$,
\item[(II)] $v=e\cdot(C\overline{C},1,0\mid A,\overline{CA},C)$ for some $e\in\F_q^*$  and octonions $A,C$ such that $A\overline{A}=0$, $C\overline{C}+1=0$,
\item[(III)] $v=(0,0,0\mid D,E,F)$, where $D,E,F$ are octonions such that
    \begin{equation}\label{eqn_IIIpt}
      D\overline{D}=E\overline{E}=F\overline{F}=0,\quad DE=EF=FD=0.
    \end{equation}
\end{enumerate}
In particular, $M_1$ is $\F_q^*$-invariant and it contains the singular vector $(0,0,0\mid x_1,0,0)$. It follows that $\cM_1$ has size $(q^4+1)\frac{q^{12}-1}{q-1}$ and consists of singular points of $Q^-(25,q)$.

We claim that $\cM_1$ is a $(q^4+1)$-tight set of $Q^-(25,q)$. By Lemma \ref{lem_basic} and the fact $\cM_1$ is single $F_4(q)$-orbit, it suffices to show that there are $q^{11}+(q^4+1)\frac{q^{11}-1}{q-1}$ points in $\cM_1$ that is perpendicular to $\la(0,0,0\mid x_1,0,0)\ra$. Suppose that $\la v\ra$ is such a point in $\cM_1$, and we examine the three types (I)-(III) one by one. If $v=(\lambda,\lambda',\lambda''\mid D,E,F)$, then the condition that $\la v\ra$ is perpendicular to $\la(0,0,0\mid x_1,0,0)\ra$ translates to $\tr(x_1\overline{D})=0$.

First consider the case where $\la v\ra$ is of type (I). The number of such points in $\cM_1$ that are perpendicular to $\la(0,0,0\mid x_1,0,0)\ra$ equals the number of $(A,B)$ pairs such that $A\overline{A}+B\overline{B}+1=0$, $\tr(x_1B)=0$. If we write $B=\sum_{i=1}^8 b_ix_i$, then $\tr(x_1B)=-b_8=0$ and $B\overline{B}=b_2b_7+b_3b_6+b_4b_5$. By Lemma \ref{lem_Nka}, we deduce that the number of such $(A,B)$ pairs is $q\cdot N_7(-1)=q^7(q^7-1)$. Here, $q$ is contributed by $b_1$ and $N_7(-1)$ is as defined in Lemma \ref{lem_Nka}.

Next consider the case where $\la v\ra$ is of type (II). The number of such points in $\cM_1$ perpendicular to $\la(0,0,0\mid x_1,0,0)\ra$ equals the number of $(A,C)$ pairs such that $C\overline{C}+1=0$, $A\overline{A}=0$ and $\tr(x_1\overline{A})=0$. By Lemma \ref{lem_ONorm}, it equals $(q^7-q^3)(q^6-q^3+q^4)$.

Finally, consider the case where $\la v\ra$ is of type (III), i.e., $v=(0,0,0\mid D,E,F)$ such that \eqref{eqn_IIIpt} holds and $\tr(x_1\overline{D})=0$. We observe that the conditions in \eqref{eqn_IIIpt} are invariant under the cyclic shift of $(D,E,F)$. We divide into the following three cases.
\begin{enumerate}
\item[(1)]Suppose that $D,E,F$ are all nonzero. There are $(q^6-q^3+q^4-1)(q^4-1)$ of $(D,E)$ pairs of nonzero octonions such that $DE=0$, $\tr(x_1\overline{D})=0$ by Lemma \ref{lem_DEpair}. For a pair $(D,E)$ of nonzero octonions such that $DE=0$, there are $q^3-1$ nonzero $F$'s such that $EF=FD=0$ by Lemma \ref{lem_ONorm}, and we deduce that $D\overline{D}=E\overline{E}=F\overline{F}=0$ for such $D,E,F$ by Lemma \ref{lem_condO}. This case contributes $\frac{q^4-1}{q-1}(q^6-q^3+q^4-1)(q^3-1)$.
\item[(2)]Suppose that exactly one of $D,E,F$ is zero. If $F=0$, then \eqref{eqn_IIIpt} reduces to $DE=0$ by Lemma \ref{lem_condO}. There are $(q^6-q^3+q^4-1)(q^4-1)$ of $(D,E)$ pairs of nonzero octonions such that $DE=0$, $\tr(x_1\overline{D})=0$ by Lemma \ref{lem_DEpair}. By symmetry, we obtain the same number if $E=0$. If $D=0$, then  \eqref{eqn_IIIpt} reduces to $EF=0$ and similarly there are $(q^3+1)(q^4-1)^2$ of such $(E,F)$ pairs of nonzero octonions that $EF=0$ by Lemmas \ref{lem_condO} and \ref{lem_ONorm}. To sum up, this case contributes $\frac{q^4-1}{q-1}(2(q^6-q^3+q^4-1)+(q^3+1)(q^4-1))$.
\item[(3)]Suppose that exactly one of $D,E,F$ is nonzero. If $E=F=0$, then \eqref{eqn_IIIpt} reduces to $D\overline{D}=0$. There are $q^6-q^3+q^4-1$ nonzero $D$'s such that $D\overline{D}=0$ and $\tr(x_1\overline{D})=0$ by Lemma \ref{lem_ONorm}. If $D=E=0$, then \eqref{eqn_IIIpt} reduces to $F\overline{F}=0$ and there are $(q^3+1)(q^4-1)$ such nonzero $F$'s. By symmetry we obtain the same number if $D=F=0$. To sum up, this case contributes $\frac{1}{q-1}\left(q^6-q^3+q^4-1+2(q^3+1)(q^4-1)\right)$.
\end{enumerate}
By adding up all the numbers, we deduce that there are exactly $q^{11}+(q^4+1)\frac{q^{11}-1}{q-1}$ points in $\cM_1$ that is perpendicular to $\la(0,0,0\mid x_1,0,0)\ra$. This completes the proof of Theorem \ref{thm_main}.

\begin{remark}\label{remark_orbits}
Suppose that the characteristic $p$ of $\F_q$ is not $3$. Take the same notation as in Section 2, and let $K$ be the normalizer of $F_4(q)$ in the normalizer of the quadratic form $Q_0$ on $U$. The vectors $x=w_1$, $y=w_4+w_4'$, $z=w_0-w_0''+w_4-w_5$ are singular vectors of $Q_0$ in $U$, and they are stabilized by the Frobenius automorphism that raises the coefficients of a vector with respect to the basis $\{w_i,w_i',w_i'':\,0\le i\le 8\}$ to their $p$-th powers. Moreover, they are white, gray and black vectors respectively and lie in distinct $\tilde{E}_6(q)$-orbits. We conclude that $K$ has at least three orbits on the singular points of the polar space associated with $(U,Q_0)$. When $p=2$, we verify with Magma \cite{magma} that there are exactly three $F_4(q)$-orbits on the singular point and each orbit form a tight set. It is not clear whether this holds for further values of $q$ such that $q\equiv 2\pmod{3}$, and we leave it as an open problem.
\end{remark}


\bibliographystyle{plain}

\end{document}